\documentclass[12pt]{amsart}

\usepackage{graphics, amsmath, amsfonts, amssymb, amscd, mathrsfs}

\newtheorem{theorem}{Theorem}
\newtheorem{proposition}[theorem]{Proposition}

\newtheorem{corollary}[theorem]{Corollary}

\theoremstyle{definition}

\newcommand{\C}{\mathbb{C}}
\newcommand{\D}{\mathbb{D}}
\renewcommand{\O}{\mathscr{O}}
\newcommand{\R}{\mathbb{R}} 

\newcommand{\T}{\mathbb{T}}

\title{Subordination principle for Proper functions}

\author{Jan Wiegerinck }
\address{Science Park  904, Amsterdam} 
\email{J.J.O.O.Wiegerinck@uva.nl}

\author{Djire Ibrahim K.}
\address{Jagiellonian University, Department of Mathematics}
\email{Ibrahim.Djire@im.uj.edu.pl}


\begin{document}

\begin{abstract} 

 Our gaol is to generalize Littlewood's Subordination Theorem to the situations where the functions are not globally subordinate. On using our result we establish a relation  between the moduli of the zeros  of the subordinate function  and the moduli of the zeros of the superordinate function. This fact has a consequence on the envelope of Lelong functional. At the end we give some coefficient inequalities of  the subordinate and superordinate functions.

\end{abstract}

\maketitle

\section{Introduction}

\smallskip\noindent
Let $f$ and $g$ be two holomorphic functions, with $f(0)=g(0)$. Suppose that, $f$ is bijective and   $g(\mathbb{D})\subset f(\mathbb{D}).$ Then $\omega(z)=f^{-1}(g(z))$ is analytic in $\mathbb{D},$ $\omega(0)=0,$ $|\omega(z)|<1$ and $g(z)=f\circ\omega(z).$ In general, an analytic function $g$ is  said to be subordinate to an analytic function $f$ if  $g(z)=f(\omega(z)),$  $|z|<1,$ for some analytic function $\omega$ with $|\omega(z)|\leq |z|.$ The superordinate function $f$ need not be univalent.
Littlewood's Subordination principle states that if $f$ and $g$ are analytic in the unit disk and if $g$ is subordinate to $f$, then, for $0<p<\infty$, $0<r<1$ we have: $$\int_0^{2\pi}|g(re^{i\theta})|^p d\theta\leq \int_0^{2\pi}|f(re^{i\theta})|^p d\theta.$$
Under its general form we have for any subharmonic function $u$ in some neighborhood of $f(\mathbb{D}) :$ $$\int_0^{2\pi}u\circ g(re^{i\theta}) d\theta\leq \int_0^{2\pi} u\circ f(re^{i\theta}) d\theta.$$
For more detail see $[1, chap.6].$

\section{Generalization}

\smallskip\noindent
We start by establishing some notation. For $r>0$, let  $D_r=\{z\in \C, |z|<r\}$ and $\D=D_1.$ For $X\subset \C$ open, $SH(X)$ denotes the set of all subharmonic functions on $X$ and $\O(D_r,X)$ the set of all holomorphic functions from $D_r$ to $X.$ An element $f\in\O(\D,X)$ is sometimes called an analytic disc in $X$ of center $f(0).$ The following theorem generalizes Littlewood's Subordination principle to the situations where $g$ is not subordinate to $f.$ 

\begin{theorem}
 Let $X$ be an open subset in $\mathbb{C}$,   and $f \in \O(\D,X)\cap C(\overline{\D},X)$ proper, then for all $u\in SH(X)$ such that $|u\circ f|$ is bounded on $\T$, for all  $g\in \O(\D,X)$ with $g(0)=f(0)$ and $g(\mathbb{D})\subset f(\mathbb{D})$ one has $$\int_0^{2\pi}u\circ g(re^{i\theta})d\theta\leq \int_0^{2\pi}u\circ f(e^{i\theta})d\theta \mbox{       for all        }r\in]0,1[.$$
\end{theorem}
\smallskip\noindent
Before proving this theorem we will recall  a  result concerning proper holomorphic functions.
\begin{proposition} 
 Let $f\in \O(\D,\mathbb{C})$ proper, $w\in \Omega=f(\D)$,  then the set $f^{-1}(w)$ is  finite and $\D\cap f'^{-1}(0)$ is countable.
\end{proposition}

\smallskip\noindent
Under the assumptions in Theorem 1 we will use Perron method for the Dirichlet Problem on  the closed unit disk $\overline{\mathbb{D}}$  to define a superharmonic function $V$ on the open set  $f(\D)$ which is greater than $u$ on $f(\D)$ and we use the properties of superharmonic functions to infer the inequality in  Theorem 1.

\begin{proof}[Proof of Theorem 1]
Let $\epsilon>0$, $g$ , $f$ and $u$ be as in Theorem 1.  As $u$ is upper semicontinuous, then  there is a sequence of continuous functions $(\psi_j)_j$ $\subset C(X)$ which decreases to $u$.  Set  $M=\max\{\psi_1\circ f(t), t\in \T\}.$ As $|u\circ f(t)|$ is bounded on $\T$, then the function $[0,2\pi] \ni\theta\rightarrow g(\theta)=\max\{|u\circ f(e^{i\theta})|, |M|\}$ is integrable on $[0,2\pi].$ Notice that $|\psi_j\circ f(e^{i\theta})|\leq g(\theta)$ for all $\theta\in [0, 2\pi]$ and $j>0$. Then by Lebesgue dominated convergence theorem there is $j_0$ such that:
$$\int_0^{2\pi}\psi_j\circ f(e^{i\theta})d\theta< \int_0^{2\pi}u\circ f(e^{i\theta})d\theta+\epsilon$$ for any $j>j_0$. We set $\psi= \psi_{j_0+1}.$ Then $\int_0^{2\pi}\psi \circ f(e^{i\theta})d\theta< \int_0^{2\pi}u \circ f(e^{i\theta})d\theta+\epsilon$.  We will define a superharmonic function $V$ on  $f(\D)$  such that $u\leq V$ on $f(\D)$.
Set $$v(t)=\sup\{w(t), w\in SH(\D),  w^*(t)\leq \psi \circ f (t), t\in \mathbb{T}\},$$ then, $v$ is harmonic on $\D$  and $v(t)=\psi \circ f(t)$ for $t\in \mathbb{T}$ see $[3, corol. 4.1.8]$.  Set $\Omega=f(\D)$ and define $v_1:\Omega\longrightarrow \mathbb{R}$ by 
 $$v_1(x)=\mbox{max} \{-v(t), t\in f^{-1}(x)\}$$ 
 where  $x\in \Omega$ see $[2]$. By Proposition 2,  $f^{-1}({x})$ is finite  so   $v_1$ makes sense. We will prove that $v_1$ is subharmonic on $\Omega$. Let $x\in\Omega\setminus f(\D\cap f'^{-1}(0))$,  then, by Local Inversion Theorem  there exist a number $k\in \mathbb{N}$,  a neighborhood $W$ of $x$ , disjoint neighbourhoods  
 $W_1$ ,..., $W_k$    of $ \{t_1,..., t_k\}=f^{-1}(x)$ and holomorphic functions $ g_i : W\rightarrow W_i $ such that $ f\circ g_i(y)=y  $ for all $ y\in W $. Then for all $y\in W$ we have 
 $$v_1(y)=\mbox{max} \{-v\circ g_i(y),  i=1,.., k\}.$$
 Hence $v_1$ is subharmonic on $W$, this for  all $x\in\Omega\setminus f( \D\cap f'^{-1}(0)).$ As subharmonicity is  a local property then $v_1$ is subharmonic on $\Omega\setminus f(\D\cap f'^{-1}(0))$. 
As $v_1$ is bounded and $ f(\D\cap f'^{-1}(0))$  is polar see $[3, corol. 3.2.5]$, then by Removable Singularity Theorem, $v_1$ can be extended to a  subharmonic function on $\Omega$.  One can find a similar work in   $[4, page 73]$.

\smallskip\noindent
We define a superharmonic function $V$ on $\Omega$ on setting  $V=-v_1.$ Notice that 
 $$V(x)=\mbox{min} \{v(t), t\in f^{-1}(x)\} \mbox{ and that } u\leq V \mbox{ on } \Omega.$$

\smallskip\noindent
We  also have  $ u\circ f(t)\leq V\circ f(t)\leq v(t)$ for any $t\in \mathbb{D}$. As  $V$ is superharmonic, $v$ is harmonic and $g(\D)\subset\Omega$, then  we have
 $$ \frac{1}{2\pi}\int_0^{2\pi} u\circ g(re^{i\theta})d\theta \leq   \frac{1}{2\pi}\int_0^{2\pi}V\circ g(re^{i\theta})d\theta\leq V\circ g(0)\leq v(0) =\frac{1}{2\pi}\int_0^{2\pi}\psi \circ f(e^{i\theta})d\theta .$$
Hence by the choice of $\psi$ we get 
$$
\int_0^{2\pi} u\circ g(re^{i\theta})d\theta \leq  \int_0^{2\pi}\psi \circ f(e^{i\theta})d\theta < \int_0^{2\pi}u\circ f(e^{i\theta})d\theta+\epsilon.$$
This for any $\epsilon >0$ hence $\int_0^{2\pi}u\circ g(re^{i\theta})d\theta\leq \int_0^{2\pi}u\circ f(e^{i\theta})d\theta.$
\end{proof} 

\smallskip\noindent
Assume that $f$ is proper then our Theorem 1 generalizes Littlewood's Subordination theorem to the situations  where $g$ is not subordinate to $f$. So on taking $u=|.|^p$, $\infty>p>0$ we get $$\int_0^{2\pi}|g(e^{i\theta})|^p d\theta\leq \int_0^{2\pi}|f(e^{i\theta})|^p d\theta.$$
One has the following interesting consequence. If two proper holomorphic functions $g$ and $f$ are such that $g(0)=f(0)$ and $g(\D)=f(\D)$,  then they have the same hardy norm. From now we say that $g\in\O(\D,\C)$ is subordinate to $f\in\O(\D,\C)$ if $g(0)=f(0)$ and $g(\D)\subset f(\D).$The following theorem gives a relation between the zero sets of subordinate and  proper superordinate functions.

\begin{theorem}
Let   $f \in \O(\D,\C)\cap C(\overline{\D},\C)$ proper. Then for all $g\in\O(\D,\C)$  such that $g(\D)\subset f(\D)$ and $g(0)=f(0)$  we have
$$ \sum_{b\in\D\cap g^{-1}(0)} m_b\log |b| \geq   \sum_{a\in f^{-1}(0)} m_a\log |a| ,$$
moreover if $g$ is proper and $g(\D)=f(\D)$ then we get 
$$  \sum_{b\in  g^{-1}(0)}m_b \log |b| =   \sum_{a\in f^{-1}(0)} m_a\log |a|.$$
With the convention  $ \sum_{a\in f^{-1}(0)}m_a \log |a|=0$ if $f$ doesn't vanish.
\end{theorem}

\begin{proof}  If $g(0)=f(0)=0$. Then we get $ -\infty= \sum_{b\in \D\cap g^{-1}(0)} m_b\log |b|= \sum_{a\in f^{-1}(0)} m_a \log|a|.$  
If $\T\cap f^{-1}(0)$ is not empty,  then all the zeros are on $\T$ hence
$$0=\sum_{b\in \D\cap g^{-1}(0)} m_b\log |b|= \sum_{a\in f^{-1}(0)} m_a\log |a| .$$
Assume that $f(0)=g(0)\neq 0$ and $\T\cap f^{-1}(0)$ is empty, then by Jesen's formula we have
$$\log |f(0)|+ \sum_{a\in f^{-1}(0)} m_a\log \frac{1}{|a|}= \int_0^{2\pi}\log | f(e^{i\theta})|d\theta$$
$$\log |g(0)|+ \sum_{b\in D_r\cap g^{-1}(0)} m_b\log \frac{r}{|b|}= \int_0^{2\pi}\log | g(re^{i\theta})|d\theta, \mbox{  for } 0\leq r<1.$$
By Theorem 1 we get 
$$\log |g(0)|+ \sum_{b\in D_r\cap g^{-1}(0)} m_b\log \frac{r}{|b|}\leq   \log |f(0)|+ \sum_{a\in f^{-1}(0)} m_a\log \frac{1}{|a|} . $$
As $g(0)=f(0)\neq 0$ then for all $r\in ]0,1[$ we have
$$ \sum_{b\in D_r\cap g^{-1}(0)} m_b\log \frac{r}{|b|}\leq  \sum_{a\in f^{-1}(0)} m_a \log \frac{1}{|a|} $$
hence 
$$ \sum_{b\in \D\cap g^{-1}(0)} m_b\log \frac{1}{|b|}\leq  \sum_{a\in f^{-1}(0)} m_a\log \frac{1}{|a|} .$$
\end{proof}

\smallskip\noindent
It is well known for a continuous function $f:\D\rightarrow \D$ that, if  $f(0)\neq 0$ then there is $r>0$ depending on $f$ such that $f(z)\neq 0$ for $z\in D(0,r)$. Here we will prove that for holomorphic functions  the choice of $r$  may not depend of $f$.

\begin{corollary}
Let $g\in \O(\D,\D)$ with $g(0)\neq 0$ then, $g$ doesn't vanish in  $D(0, |g(0)|)$.
\end{corollary}
\begin{proof}
Set $a_0=g(0)$ and $f(z)=\frac{z+a_0}{1+\bar{a}_0 z}$ then $f$ is proper and has a simple zero at $-a_0$. Notice that  $g$ is subordinate to $f$. Now assume that there is $b_0\in D(0,|g(0)|)$ such that $g(b_0)=0,$ then 
 $$         \sum_{a\in f^{-1}(0)} m_a\log |a|\geq
 \log|-a_0|>\log|b_0|\geq \sum_{b\in\D\cap g^{-1}(0)} m_b\log |b|.$$
This is in contradiction with the theorem above.
\end{proof}

\smallskip\noindent
For $x\in\D^*$ we denote by $\O(\D,\D,x)$ the set of all holomorphic functions $f\in\O(\D,\D)$ with $f(0)=x$, then $D(0,|x|)$ is the largest disk  in which none  element of $\O(\D,\D,x)$ vanishes. Notice that Corollary 4 can be seen as a consequence of Schwarz Lemma and Corollary 6 a generalization of Corollary 4.

\begin{corollary}
Let $X\subset\C$ be open and $f \in \O(\D,X)\cap C(\overline{\D},X)$ be proper. Then for all $g\in\O(\D,X)$  such that $g(\D)\subset f(\D)$ and $g(0)=f(0)$  we have
$$ \sum_{b\in g^{-1}(p)} m_b\log |b| \geq   \sum_{a\in f^{-1}(p)} m_a\log |a| , \mbox{   for   } p\in X.$$
With the convention  $ \sum_{a\in f^{-1}(p)}m_a \log |a|=0$ if the function $f-p$ doesn't vanish.
\end{corollary}

\smallskip\noindent
The corollary  above gives an idea about the location  of the solutions of certain  equations. For instance we have.
\begin{corollary}
Let $g\in \O(\D,\D,x)$ with $x\in\D^{*}$ then, the function  $g-a$ doesn't vanish in  the disk $D(0,\frac{|a-x|}{|1-\bar{x} a|})$ for  all $a\in\D\setminus \{x\}$.
\end{corollary}

\subsection{ Lelong functional in $\C$}
Consider $\alpha(x)=\sum_{j=1}^{N}m_j\chi_{\{p_j\}}(x)$, where $p_j$ are points in $X$ and $m_j$ are positive, then we define a disc functional
$$ H_{\alpha}^L:\O(\D,X)\rightarrow\R\cup\{-\infty\}, \mbox{           } H^L_{\alpha}(f)=\sum_{a\in \D}\alpha(f(a))m_a\log|a|,$$
which is called the Lelong functional with respect to $\alpha.$  Its envelope  is the following   $$ EH^L_{\alpha}(x)=\inf\{H_{\alpha}^L(f), f\in\O(\D,X), f(0)=x\}.$$
In $[5]$ it is proven that $EH_{\alpha}^L$ is subharmonic and it  coincides with the Green function of $X$ with several poles at $p_1,\dots,p_N$ of weights $m_1\dots, m_N$.   The theorem below states that  the $infimum$ in the definition of $EH^L_{\alpha}$ is  actually a $minimum$ in a special case. 

\begin{theorem} Let $X$ be open in $\C$. If there is $F\in \O(\D,\C)$ proper such that $F(\D)=X,$ then  for all $x\in X$ there is $f\in\O(\D,X)$ with $f(0)=x$  such that 
$$EH_{\alpha}^{L}(x)= H_{\alpha}^L(f)=\sum_{z\in\D} \alpha( f(z)) m_z\log|z|.$$
\end{theorem}

\begin{proof}
Let $x\in X$,  $\psi$ be an automorphism of $\D$ such that $F\circ\psi(0)=x$ and set $f=F\circ\psi.$ For all $g\in\O(\D,X)$ with $g(0)=x$ we have 
 $$H_{\alpha}^L(f)=  \sum_{a\in \D}\alpha(f(a)) m_a\log|a|=\sum_{p\in X}\alpha(p)\sum_{a\in f^{-1}(p)} m_a\log|a|.$$
$$H_{\alpha}^L(g)=  \sum_{b\in \D}\alpha(g(b))m_b\log|b|=\sum_{p\in X}\alpha(p)\sum_{b\in g^{-1}(p)} m_b\log|b|.$$
Then by Corollary 5 we get $ H_{\alpha}^L(f) \leq H_{\alpha}^L(g)$  hence
 $$ EH^L_{\alpha}(x)\leq     H_{\alpha}^L(f)  \leq \inf\{H_{\alpha}^L(g), g\in\O(\D,X), g(0)=x\}   \leq EH^L_{\alpha}(x).   $$
\end{proof}

\smallskip\noindent
Let $z_x\in F^{-1}(x)$ and set $\psi_x(z)=\frac{z+z_x}{1+\bar{z}_xz}$ on taking $\psi_x(z)=a$ we get 
$$EH^L_{\alpha}(x)=    \sum_{a\in\D} \alpha( F(a)) m_a \log \left|\frac{a-z_x}{1-\bar{z}_x a}\right|.$$
Remark that the value of $EH^L_{\alpha}(x)$ doesn't depend on the choice of $z_x\in F^{-1}(x)$, $x\in X.$

\subsubsection{ Lelong functional in $\C^n$} Let $X\subset\C^n$  be open, consider $\alpha(x)=m\chi_{\{p\}}(x)$, where $p$ is a point in $X$ and $m$ is nonnegative.
\begin{theorem}
Assume that $X=X_1\times\dots\times X_n$ and $X_i\subset\C$, $i=1,...,n$ are jordan domains, then
at each point $x\in X$ there is an extremal disc for $EH^L_{\alpha}$ in other word there is  $f\in\O(\D,X)$ with $f(0)=x$ such that $$ EH^L_{\alpha}(x)=H^L_{\alpha}(f)=    \sum_{a\in \D}\alpha(f(a)) m_a\log|a|.    $$
\end{theorem}
\begin{proof}
Denote $F_i$ the Riemann mapping from $\D$ onto $X_i$, for $x\in X$ we consider $\psi_i$ an automorphism of $\D$ such that $F_i\circ\psi_i(0)=x_i$. Take $z_i\in\D$ such that $F_i\circ\psi_i(z_i)=p_i$. Take $a\in\{z_1,\dots, z_n\}$ such that $|a|=\max\{|z_i|,  i=1,\dots, n\}$. We obtain  an analytic disc $f$ in $X$ centered at $x$ containing $p$ on setting $f(z)=(F_1\circ\psi_1(\frac{z_1}{a}z),\dots,F_n\circ\psi_n(\frac{z_n}{a}z))$. We may assume that $a=z_1$. Let $g\in\O(\D,X)$ with $g(0)=x$ remark that  $$\sum_{a\in g^{-1}(p)} m_a\log |a|    \geq   \sum_{a\in g_1^{-1}(p_1)} m_a\log |a| \geq   \sum_{a\in f_1^{-1}(p_1)} m_a\log |a| =   \sum_{a\in f^{-1}(p)} \log |a|.$$ Hence $H_{\alpha}^L(g)\geq H_{\alpha}^L(f)$ for all $g\in\O(\D, X, x)$ then $EH_{\alpha}^L(x)=H^L_{\alpha}(f)=\alpha(f(a))\log|a|$ where $f(a)=p.$
\end{proof}

\smallskip\noindent
Remark that on setting $z_i=F^{-1}(p_i)$ we can write $EH^L_{\alpha}$ as follow $$EH^L_{\alpha}(x)=\alpha(p)\max\left\{\log \left| \frac{F^{-1}(x_i)-z_i}{1-\bar z_i  F^{-1}(x_i)} \right|,    i=1,\dots, n \right\}.$$

\begin{corollary}
Let $g$ and $f$ be two closed proper analytic discs in $X$, where $X\subset \mathbb{C}$ is open.  Assume that $g(0)=f(0)$ and $g(\mathbb{\overline{D}})=f(\mathbb{\overline{D}}),$ then for any $u\in SH(X)$ bounded one has: $$\int_0^{2\pi}u\circ g(e^{i\theta}) d\theta=\int_0^{2\pi}u\circ f(e^{i\theta}) d\theta.$$
\end{corollary}

\smallskip\noindent
\subsection{Coefficient Inequalities}
If $g(z)=\sum b_nz^n$ is subordinate to a proper function  $f(z)=\sum a_nz^n,$ then  the coefficients of $f$ dominate those of $g$ in a certain average sense.

\begin{corollary}
Let   $f(z)=\sum_{n=1}^{\infty}a_nz^n$ be   proper analytic  in $\D$, continuous on $\overline{\D}$ and  $g(z)=\sum_{n=1}^{\infty}b_nz^n$ be analytic in $\D$.  Suppose $g(\overline{\D})\subset f(\overline{\D})$ and $b_1\neq 0$. Then there exist $N>2$ such that 
$$ \sum_{k=2}^{n} |b_k|^2\leq  \sum_{k=1}^{n} |a_k|^2  \mbox{,        for all           }       n= N+1,\dots . $$
\end{corollary}

\begin{proof}
Set $A_n=\sum_{k=1}^{n} |a_k|^2 $ and $B_n=   \sum_{k=1}^{n} |b_k|^2$. By Parseval's relation,
$$ \frac{1}{2\pi}\int_0^{2\pi}|f(e^{i\theta})|^2 d\theta=   \sum_{k=1}^{\infty} |a_k|^2=\lim_{n\rightarrow \infty} A_n .$$
So there is $N>2$ such that $  \lim_{n\rightarrow \infty} A_n\leq  A_n + |b_1|^2 $ for $n>N.$ Then by Theorem 1,
$$\sum_{k=1}^{n} |b_k|^2 \leq  \frac{1}{2\pi} \int_0^{2\pi}|g(e^{i\theta})|^2 d\theta\leq   \frac{1}{2\pi} \int_0^{2\pi}|f(e^{i\theta})|^2 d\theta\leq \sum_{k=1}^{n} |a_k|^2 + |b_1|^2,  \mbox{       for     } n>N.$$
Hence
$$\sum_{k=2}^{n} |b_k|^2  \leq \sum_{k=1}^{n} |a_k|^2         ,  \mbox{       for     } n>N.$$
\end{proof}
\begin{corollary}
If  $a_n= O(1),$ then $ b_n=O( \sqrt{n})$ as $n \rightarrow \infty.$
\end{corollary}
\smallskip\noindent
Here we compare the area of $f(D_r)$ with the area of $g(D_{r^2})$ for $r$ small.
\begin{corollary}
Let   $f(z)=\sum_{n=1}^{\infty}a_nz^n$ be   proper analytic  in $D_r$,  $g(z)=\sum_{n=1}^{\infty}b_nz^n$ be analytic in $D_r$  and suppose $g({D_r})\subset f({D_r})$. Then 
$$ \int\int_{|z|\leq r^2}|g'(\rho e^{i\theta})|^2\rho  d\rho  d\theta \leq  \int\int_{|z|\leq r}|f'(\rho e^{i\theta})|^2\rho  d\rho  d\theta,   \mbox{                  }   0\leq r\leq \left(\frac{1}{e}\right)^{\frac{1}{2e}}.  $$
\end{corollary}
\begin{proof}
The integrals represent the areas of the image (the multisheeted image) of $D_{r^2}$ under $g$ and the image of $D_r$ under $f.$
By Theorem 1 we have  
$$   \sum_{n=1}^{\infty}|b_n|^2r^{2n}  \leq       \sum_{n=1}^{\infty}|a_n|^2 r^{2n}.$$
Notice that for $r\leq (\frac{1}{e})^{\frac{1}{2e}}$, the function $x\rightarrow xr^{2x}$ is less than $1$ for all $x\geq 1$. Therefore   $nr^{2n}\leq1$ for all $n>1$, then
$$   \sum_{n=1}^{\infty} nr^{2n}|b_n|^2r^{2n}    \leq \sum_{n=1}^{\infty}|b_n|^2r^{2n}  \leq       \sum_{n=1}^{\infty}|a_n|^2 r^{2n}     \leq  \sum_{n=1}^{\infty} n|a_n|^2 r^{2n}. $$
Hence
$$\sum_{n=1}^{\infty} n|b_n|^2{r^2}^{2n}    \leq \sum_{n=1}^{\infty} n|a_n|^2 r^{2n}.$$
This last inequality is equivalent to the inequality in the corollary.
\end{proof}
\smallskip\noindent
We can also compare the Hardy norm of the derivatives on $D_r$, $r\leq (\frac{1}{e})^{\frac{1}{e}}$. The proof of the corollary below is similar to that of the corollary above.
\begin{corollary}
Let   $f(z)=\sum_{n=1}^{\infty}a_nz^n$ be   proper analytic  in $D_r$,  $g(z)=\sum_{n=1}^{\infty}b_nz^n$ be analytic in $D_r$  and suppose that  $g({D_r})\subset f({D_r})$. Then 
$$ \int_{0}^{2\pi}|g'(r^2 e^{i\theta})|^2 d\theta \leq  \int_{0}^{2\pi}|f'(re^{i\theta})|^2  d\theta,   \mbox{                  }   0\leq r\leq \left(\frac{1}{e}\right)^{\frac{1}{e}}.
    $$
\end{corollary}

\smallskip\noindent
{\it Acknowledgement.}  The second author  thanks his  supervisor Professor Armen Edigarian for his help to accomplish this work.

\end{document}